\documentclass[a4paper]{article}
\usepackage[utf8]{inputenc} 
\usepackage{hyperref}
\usepackage{amsfonts,amsmath,amstext,amssymb}
\usepackage{graphicx}
\usepackage{hyperref}\hypersetup{colorlinks=true, linkcolor=blue, citecolor=blue, filecolor=blue, urlcolor=red, pdfauthor={Nome do Autor}, pdftitle={On bisectors in quaternionic hyperbolic space},          pdfsubject={Assunto do Projeto}, pdfkeywords={palavra-chave, palavra-chave, palavra-chave}, pdfproducer={Latex},        pdfcreator={pdflatex}}
\def\re{\mathrm Re\,}
\def\im{\mathrm Im\,}

\def\mod#1{\vert #1\vert}

\def\hherm#1{{\langle\!\langle #1 \rangle\!\rangle}}

\def\B{\mathbb B}

\def\R{\mathbb R}
\def\C{\mathbb C}
\def\D{\mathbb D}
\def\Z{\mathbb Z}
\def\Q{\mathbb Q}
\def\P{\mathbb P}

\def\H{\mathbb H}

\def\F{\mathbb F}

\def\herm#1{\langle #1\rangle}

\newenvironment{proof}{\noindent\normalsize {\sc Proof}:}{{\hfill \rule{1mm}{3mm}}}

\newtheorem{theorem}{Theorem}[section]
\newtheorem{co}{Corollary}[section]
\newtheorem{prop}{Proposition}[section]
\newtheorem{lemma}{Lemma}[section]
\newtheorem{rmk}{Remark}[section]

\date{}

\title{On bisectors in quaternionic hyperbolic space}

\author{Igor A.R. Almeida \\ 
{\small Departamento de Matem\'{a}tica} \\ 
{\small Universidade Federal de Minas Gerais} \\ 
{\small Belo Horizonte - MG - Brazil} \\ \\
Jaime L.O. Chamorro  \\ 
{\small Departamento de Matem\'{a}tica} \\ 
{\small Universidade Federal de Baia} \\
{\small  Salvador - Ba - Brazil}\\ \\
Nikolay Gusevskii\thanks{Corresponding author.} \\ 
{\small Departamento de Matem\'{a}tica} \\
{\small Universidade Federal de Minas Gerais} \\ 
{\small Belo Horizonte - MG - Brazil} \\ {\small e-mail: nikolay@mat.ufmg.br}}

\date{}

\begin{document}
\maketitle

%----------------------------------------------------------------------------------

\begin{abstract}

In this paper, we study a problem related to geometry of bisectors in quaternionic hyperbolic geometry. We develop some of the basic theory of bisectors in  quaternionic hyperbolic space ${\rm H^{n}_{\Q}}$.
In particular, we show that quaternionic bisectors enjoy various decompositions  by totally geodesic submanifolds of ${\rm H^{n}_{\Q}}$.  In contrast to complex hyperbolic
geometry, where bisectors admit only two types of decomposition  (described by Mostow and Goldman), we show that in the quaternionic case  geometry of bisectors is more rich.
The main purpose of the paper is to describe an infinite family of different decompositions of bisectors  in ${\rm H^{n}_{\Q}}$ by totally geodesic submanifolds of ${\rm H^{n}_{\Q}}$ isometric to complex
hyperbolic space ${\rm H^{n}_{\C}}$ which we call the fan decompositions.  Also, we derive a formula for the orthogonal projection onto totally geodesic submanifolds in ${\rm H^{n}_{\Q}}$ isometric to  ${\rm H^{n}_{\C}}$. Using this, we introduce a new class of hypersurfaces in ${\rm H^{n}_{\Q}}$, which we call complex hyperbolic packs in ${\rm H^{n}_{\Q}}$. We hope that the complex hyperbolic packs will be useful for constructing fundamental polyhedra for discrete groups of isometries of quaternionic hyperbolic space.
\end{abstract}



\section*{Introduction}

\medskip

It is well known that the rank one symmetric spaces of non-compact
type are real, complex and quaternionic hyperbolic spaces, or a Cayley hyperbolic plane. In order to construct discrete groups of isometries in these geometries, one needs an appropriate
notion of polyhedra, which becomes non-trivial in the absence totally geodesic hypersurfaces from which to form  a polyhedron's faces
in spaces with non-constant sectional curvature. This happens in complex and quaternionic hyperbolic geometries, and in geometry of the Cayley hyperbolic plane. Since the faces of
Dirichlet fundamental polyhedra are in bisectors, it is natural to use bisectors as the building blocks for polyhedra in  these geometries. This idea goes back to Giraud \cite{Gir} and was developed further by Mostow \cite{Mos} and Goldman  \cite{Gol}, (see also \cite{AGG}, \cite{FP}, \cite{GuP1}, \cite{GuP2}, \cite{DPP} for other examples of fundamental domains bounded by bisectors).

\medskip

Therefore, it is necessary to understand the geometric  structure of such hypersurfaces. In complex hyperbolic geometry, it was done by Mostow \cite{Mos} and Goldman  \cite{Gol}. They showed that bisectors in complex hyperbolic space ${\rm H^{n}_{\C}}$ admit a decomposition into complex hyperplanes and into totally real totally geodesic submanifolds (meridian decomposition). Goldman also proved that these decompositions are unique. 

\medskip

In this paper, we will show that some basic results about bisectors from complex hyperbolic
geometry carry over  to the quaternionic case. But, it will be shown that the geometry of bisectors in quaternionic hyperbolic space is more rich.
First, we prove an analogue
of the Mostow decomposition of bisectors in ${\rm H^{n}_{\Q}}$ and its generalization. Then we  show that any bisector in quaternionic hyperbolic space ${\rm H^{n}_{\Q}}$ is a union of totally geodesic submanifolds of ${\rm H^{n}_{\Q}}$
isometric to ${\rm H^{n}_{\C}}$ intersecting in a common point. In some sense, this is an analogue of the Goldman meridian decomposition. We call such decompositions the fan decompositions. We will show that such decompositions are not unique. The existence of fan decompositions implies, in particular, that any bisector in quaternionic hyperbolic space is star-like with respect to any point in its real spine.  Also, we derive a formula for the orthogonal projection onto totally geodesic submanifolds in ${\rm H^{n}_{\Q}}$ isometric to  ${\rm H^{n}_{\C}}$. Using this, we introduce a new class of hypersurfaces in ${\rm H^{n}_{\Q}}$, which we call complex hyperbolic packs in ${\rm H^{n}_{\Q}}$. The construction of complex hyperbolic packs in ${\rm H^{n}_{\Q}}$ in somewhat similar to the construction of packs introduced by Parker-Platis, see \cite{PP}, in complex hyperbolic geometry. We hope that the complex hyperbolic packs will be useful for constructing fundamental polyhedra for discrete groups of isometries of quaternionic hyperbolic space.

\medskip

The paper is organized as follows. In Section 1,  we summarize some basic results about geometry of quaternionic hyperbolic space.  In Section 2, we study geometry of bisectors in quaternionic hyperbolic space. In Section 3, we derive a formula for the orthogonal projection onto totally geodesic submanifolds of ${\rm H^{n}_{\Q}}$ isometric to  ${\rm H^{n}_{\C}}$ and introduce a new class of hypersurfaces in ${\rm H^{n}_{\Q}}$, which we call complex hyperbolic packs in ${\rm H^{n}_{\Q}}$.

\section{Preliminaries}

\noindent In this section, we recall some basic results  related to quarternions and geometry of projective and hyperbolic spaces.

\subsection{Quaternions}

\noindent First, we recall some basic facts about the quaternions we need. The quaternions $\Q$ are the $\R$-algebra generated by the symbols $i,j,k$ with the relations

$$ l^2=j^2=k^2=-1,  \ \ \  ij=-ji=k, \ \ \  jk=-kj=1, \ \ \  ki=-ik=j.$$

\medskip

So, $\Q$ is a skew field and a 4-dimensional division algebra over the reals.

\medskip

Let $a \in \Q$. We write $a=a_0 + a_1 i + a_2 j + a_3 k$, $a_i \in \R$,
then by definition

$$\bar{a}= a_0 - a_1 i - a_2 j - a_3 k, \ \ \ {\rm Re} \ a=a_0, \ \ \ {\rm Im} \ a=a_1 i + a_2 j + a_3 k.$$

\medskip

Note that, in contrast with the complex numbers, ${\rm Im} \ a$ is not a real number (if $a_i\neq 0$ for some $i=1,2,3$), and that conjugation obeys the rule

$$\overline{ab} = \bar{b}\bar{a}.$$

\medskip

Also, we define $\mod{a}=\sqrt{a\bar{a}}$. We have that if $a\neq 0$ then $a^{-1}=\bar{a}/\mod{a}^2.$

\medskip

In what follows, we will identify the reals  numbers $\R$ with ${\R} 1$ and the complex numbers $\C$ with the subfield of $\Q$ generated over $\R$ by $1$ and $i$.

\medskip

Two quaternions $a$ and $b$ are called {\sl similar} if there exists $\lambda \neq 0$ such that $a=\lambda b \lambda^{-1}.$
By replacing $\lambda$ by $\tau=\lambda/\mod{\lambda}$, we may always assume $\lambda$ to be unitary.

\medskip

The following proposition was proved in \cite{Bre}.

\begin{prop}
\label{qsim} 
Two quaternions $a$ and $b$ are similar if and only if ${\rm Re} \ a ={\rm Re} \ b$ and
$\mod{a} = \mod{b}$. Moreover, every similarity class contains a complex number, unique up to conjugation.
\end{prop}

\begin{co} \label{qua1} Any quaternion $a$ is similar to a unique complex number $b=b_0 +b_1 i $ such that $b_1 \geq 0.$
\end{co}

Also, this proposition implies that every quaternion is similar to its conjugate.

\medskip

\noindent {\bf Example:}  $j z j^{-1} = \bar{z}$ for all $z \in \C$.

\medskip

We say that $a\in \Q$ is {\sl purely imaginary} if ${\rm Re}(a)=0 $. Let us suppose that $a$ is purely imaginary and $|a|=1$. Then $a^2=-1$. This implies that the real span of $1$ and $a$ is a subfield of $\Q$ isomorphic
to the field of complex numbers. We denote this subfield by $\C(a)$. It is easy to prove that any subfield of $\Q$ containing real numbers and isomorphic to the field of complex numbers is of the form $\C(a)$
for some $a$ purely imaginary with $|a|=1$.

\medskip

The following was proved in \cite{CGr}.

\medskip

\begin{prop}\label{qua2}Let $a$ be as above. Then the centralizer of $a$ is $\C(a)$.
\end{prop}

\medskip

More generally, for any $\lambda \in \Q$, let $\C(\lambda)$ also denote the real span of $1$ and $\lambda$.

\begin{prop} \label{qua3} Let $\lambda \in \Q \setminus \R$. Then the centralizer of $\lambda$ is $\C(\lambda)$.
\end{prop}

\subsection{Hyperbolic spaces}

\noindent In this section, we discuss two models for the hyperbolic spaces, its isometry group and totally geodesics submanifolds.

\subsubsection{Projective Model}

\noindent We denote by $\F$ one of the real division algebras $\R$, $\C$, or $\Q$.
Let us write $\F ^{n+1}$ for a right $\F$ - vector space of dimension $n+1$.
The $\F$- projective space $\P\F^{n}$ is the manifold of right $\F$-lines in
$\F^{n+1}$. Let $\pi$ denote a natural projection
from $\F^{n+1} \setminus \{0\}$ to the projective space $\P\F^{n}$.

\medskip

Let $\F^{n,1}$ denote a $(n+1)$-dimensional $\F$-vector space
equipped with a Hermitian form $\Psi= \herm{-,-}$ of signature $(n,1)$.
Then there exists a (right) basis in $\F ^{n+1}$ such that the Hermitian
product is given by
$\herm{v,w}=v^*J_{n+1}w,$ where $v^*$ is the Hermitian transpose of
$v$ and  $J_{n+1}=(a_{ij})$ is the $(n+1)\times (n+1)$-matrix with
$a_{ij}=0$ for all $i \neq j$, $a_{ii}=1$ for all $i=1, \ldots, n$,
and $a_{ii}=-1$ when $i=n+1$.

\medskip

That is,
$$\herm{v,w}=\bar{v}_1 w_1+\ldots +  \bar{v}_nw_n -\bar{v}_{n+1}w_{n+1},$$
where $v_i$ and $w_i$ are coordinates of $v$ and $w$ in this basis. We call such a basis in $\F^{n,1}$ an {\sl orthogonal basis} defined by
a Hermitian form $\Psi= \herm{-,-}$.

\medskip

Let $V_{-}, V_0, V_{+}$ be the subsets of $\F^{n,1} \setminus \{0\}$
consisting of vectors where $\herm{v,v}$ is negative, zero, or
positive respectively. Vectors in $V_0$ are called {\sl null} or
{\sl isotropic}, vectors in $V_{-}$ are called {\sl negative}, an
vectors in  $V_{+}$ are called {\sl positive}. Their projections to $\P\F^{n}$
 are called {\sl isotropic}, {\sl negative}, and {\sl
positive} points respectively.

\medskip

The projective model of {\sl  hyperbolic space} ${\rm H^{n}_{\F}}$ is
the set of negative points in  $\P\F^{n}$, that is,
${\rm H^{n}_{\F}}=\pi(V_{-}).$

\vspace{2mm}

We will consider ${\rm H^{n}_{\F}}$ equipped with the Bergman metric \cite{CGr}:

$$d(p,q)=\cosh^{-1}\{|\Psi(v,w)| [\Psi(v,v)\Psi(w,w)]^{-1/2}\},$$

where $p,q \in{\rm H^{n}_{\F}}$, and $\pi(v)=p, \pi(w)=q$.

\medskip

The boundary $\partial
{\rm H^{n}_{\F}}=\pi(V_{0})$ of ${\rm H^{n}_{\F}}$  is the sphere formed by
all isotropic  points.

\medskip

Let ${\rm U}(n,1; \F)$ be the unitary group
corresponding to this Hermitian form $\Phi$. If $g \in {\rm U}(n,1; \F)$, then
$g(V_{-})=V_{-}$ and $g(v\lambda)=(g(v))\lambda$, for all $\lambda \in
\F$. Therefore, ${\rm U}(n,1; \F)$ acts on $\P\F^{n}$ leaving ${\rm H^{n}_{\F}}$
invariant.

\medskip

The group ${\rm U}(n,1; \F)$ does not act effectively on ${\rm H^{n}_{\F}}$. The kernel of this
action is the center ${\rm Z}(n,1; \F)$. Thus, the projective group ${\rm PU}(n,1; \F)={\rm U}(n,1; \F)/{\rm Z}(n,1; \F)$ acts
effectively. The center $\Z(n,1, \F)$ in ${\rm U}(n,1; \F)$ is $\{\pm {\rm E}\}$ if $\F=\R$ or $\Q$, and is the circle group $\{\lambda {\rm E} : \mod \lambda =1\}$
if $\F=\C$. Here ${\rm E}$ is the identity  transformation of  $\F^{n,1}$.

\medskip

It is well-known, see for instance \cite{CGr}, that ${\rm PU}(n,1; \F)$
acts transitively on ${\rm H^{n}_{\F}}$ and doubly transitively on
$\partial{\rm H^{n}_{\F}}$.

\medskip

We remark that
\begin{itemize}
\item if $\F=\R$ then  ${\rm H^{n}_{\F}}$ is a real hyperbolic space ${\rm H^{n}_{\R}}$,

\item if $\F=\C$ then  ${\rm H^{n}_{\F}}$ is a complex hyperbolic space ${\rm H^{n}_{\C}}$,

\item if $\F=\Q$ then  ${\rm H^{n}_{\F}}$ is a quaternionic  hyperbolic space ${\rm H^{n}_{\Q}}$.
\end{itemize}

\medskip

It is easy to show \cite{CGr} that ${\rm H^{1}_{\Q}}$ is isometric to ${\rm H^{4}_{\R}}$.

\subsubsection{The ball model}\label{ball}

\noindent In this section, we consider the space $\F^{n,1}$ equipped by an orthogonal basis
$$e=\{e_1, \ldots, e_n, e_{n+1} \}.$$

For any $v \in \F^{n,1}$, we write $v=(z_1, \ldots, z_n, z_{n+1})$, where $z_i$, $i=1,\ldots, n+1$ are coordinates
of $v$ in this basis.

\medskip

If $v=(z_1, \ldots, z_n, z_{n+1}) \in V_{-}$, the condition $ \herm{v,v} < 0$ implies that $z_{n+1} \neq 0.$ Therefore, we may define a set of coordinates $w=(w_1, \ldots, w_n)$ in ${\rm H^{n}_{\F}}$ by $w_i(\pi(z))=z_i z_{n+1}^{-1}.$
In this way  ${\rm H^{n}_{\F}}$ becomes identified with the ball

$$\B=\B(\F)= \{w=(w_1, \ldots, w_n) \in \F^n : \Sigma_{i=1}^n |w_i|^2 < 1\}.$$

\medskip

With this identification the map $\pi: V_{-} \rightarrow {\rm H^{n}_{\F}}$ has the coordinate representation $\pi(z)=w$, where
$w_i= z_i z_{n+1}^{-1}$.

\medskip

Let $\F^{n}$ denote a $n$-dimensional $\F$-vector space
equipped with the standard positive definite Hermitian form.
Then there exists a (right) basis in $\F ^n$ such that the Hermitian
product $\hherm{z,w}$ of two vectors $z,w$ in $\F^n$  is given by

$$\hherm{z,w}= \overline{z}_1 w_1 + \ldots, + \overline{z}_n w_n,$$
where $z_i$, $w_i$, $i=1,\ldots, n$ are coordinates of $z$ and $w$ in this basis.

\medskip

Using this, we have that 

$$\B=\B(\F)= {w=(w_1, \ldots, w_n) \in \F^n :\hherm{w,w} <1 }.$$

\subsubsection{Totally geodesic submanifolds}\label{tgsm}

We will need the following result, see \cite{CGr}, which describes all totally
geodesic submanifolds of ${\rm H^{n}_{\F}}$.

\medskip

Let $F$ be a subfield of $\F$.  An $F$-{\sl unitary subspace} of $\F^{n,1}$ is an $F$-subspace of $\F^{n+1}$
in which the Hermitian form $\Phi$ is $F$-valued. An $F$-{\sl hyperbolic subspace} of $\F^{n,1}$ is an $F$-unitary subspace in which
the Hermitian form $\Phi$ is non-degenerate and indefinite.

\medskip

\begin{prop}\label{ptgsm}
Let $M$ be a totally geodesic submanifold of ${\rm H^{n}_{\F}}$. Then either

\vspace{2mm}

(a) $M$ is the intersection of the projectivization of an $F$-hyperbolic subspace of $\F^{n,1}$ with ${\rm H^{n}_{\F}}$ for some
subfield $F$ of $\F$, or

\vspace{2mm}

(b) $\F=\Q$, and $M$ is a 3-dimensional  totally geodesic submanifold of a totally geodesic quaternionic line  ${\rm H^{1}_{\Q}}$  in ${\rm H^{n}_{\Q}}$.
\end{prop}

\medskip

From the last proposition follows that

\begin{itemize}
\item in the real hyperbolic space ${\rm H^{n}_{\R}}$ any totally geodesic submanifold  is isometric to ${\rm H^{k}_{\R}}, \ k=1, \ldots , n,$

\item in the  complex hyperbolic space ${\rm H^{n}_{\C}} $ any totally geodesic submanifod  is isometric to ${\rm H^{k}_{\C}},$ or to
${\rm H^{k}_{\R}},$ $\ k=1, \ldots , n,$

\item in the  quaternionic  hyperbolic space ${\rm H^{n}_{\Q}}$ any totally geodesic submanifold  is isometric to ${\rm H^{k}_{\Q}},$
 or to ${\rm H^{k}_{\C}},$ or to ${\rm H^{k}_{\R}},$
$k=1, \ldots , n,$ or to a 3-dimensional  totally geodesic submanifold of a totally geodesic quaternionic line  ${\rm H^{1}_{\Q}}$.
\end{itemize}

\medskip

In what follows we will use the following terminology:

\begin{itemize}
\item A totally geodesic submanifold of ${\rm H^{n}_{\Q}}$ isometric to ${\rm H^{1}_{\Q}}$ is called a {\sl quaternionic geodesic}.
\item A totally geodesic submanifold of ${\rm H^{n}_{\Q}}$ isometric to ${\rm H^{1}_{\C}}$ is called a {\sl  complex geodesic}.
\item A totally geodesic submanifold of ${\rm H^{n}_{\Q}}$ isometric to  ${\rm H^{k}_{\R}}$ is called a {\sl $\R^k$-plane.} 
\item A totally geodesic submanifold of ${\rm H^{n}_{\Q}}$ isometric to  ${\rm H^{2}_{\R}}$ is called a {\sl real plane}.

\end{itemize}

\medskip

It is clear that two distinct points in ${\rm H^{n}_{\Q}} \cup \partial
{\rm H^{n}_{\Q}}$  span a unique quaternionic geodesic. We also remark that any $2$-dimensional totally geodesic submanifold of a totally geodesic
quaternionic line  ${\rm H^{1}_{\Q}}$ is isometric to ${\rm H^{1}_{\C}}$.

\begin{prop}
Let $V$ be a subspace of $\F^{n,1}$. Then each linear isometry of $V$ into $\F^{n,1}$ can be extended to an element of ${\rm U}(n,1; \F)$.
\end{prop}

\medskip

This is a particular case of the Witt theorem, see \cite{Sch}.

\medskip

\begin{co}
Let $S \subset {\rm H^{n}_{\F}}$ be a totally geodesic submanifold. Then each linear isometry of $S$ into ${\rm H^{n}_{\F}}$ can be extended to an
element of the isometry group of ${\rm H^{n}_{\F}}$.
\end{co}

\medskip

An interesting class of totally geodesic submanifolds of the quaternionic hyperbolic space ${\rm H^{n}_{\Q}}$ are submanifods which we call
totally geodesic submanifolds of complex type, or simply, submanifolds of complex type. Their construction is the following.
Let $\C^{n+1}(a) \subset\ \Q^{n+1} $ be the subset of vectors in $\Q^{n+1} $ with coordinates in $\C(a)$, where $a$ is a purely imaginary quaternion,
$\mod a =1$. Then $\C^{n+1}(a)$ is a vector space over the field $\C(a)$. The projectivization of $\C^{n+1}(a)$, denoted by ${\rm M}^n(\C(a))$, is a
projective submanifold of $\P\Q^{n}$ of real dimension $2n$. We call this submanifold  ${\rm M}^n(\C(a))$ a {\sl  projective submanifold of complex type of maximal dimension}.
It is clear that  the space $\C^{n+1}(a)$ is indefinite. The intersection
${\rm M}^n(\C(a))$  with ${\rm H^{n}_{\Q}}$ is a totally geodesic submanifold of  ${\rm H^{n}_{\Q}}$, called  a {\sl totally geodesic submanifold of complex type of maximal dimension}. It was proven in \cite{CGr} that all these submanifolds are isometric, and, moreover, they are globally equivalent
with respect to the isometry group of  ${\rm H^{n}_{\Q}}$, that is, for any two such submanifolds ${\rm M}$ and ${\rm N}$ there exists an element
$g \in  {\rm PU}(n,1; \Q)$ such that ${\rm M=g(N)}$. In particular, all of them are globally equivalent with respect to ${\rm PU}(n,1; \Q)$ to the
{\sl canonical totally geodesic complex  submanifold} ${\rm H^{n}_{\C}}$  defined by $\C^{n+1} \subset \Q^{n+1}$. This corresponds to the canonical
subfield of complex numbers $\C=\C(i) \subset  \Q$ in the above.

\medskip

If $V^{k+1} \subseteq  \C^{n+1}(a)$ is a subspace of complex dimension $k+1$, then its projectivization $W$ is called a {\sl projective submanifold of complex type of complex dimension k}. When $V^{k+1} \subseteq \C^{n+1}$, then its projectivization $W$ is called a {\sl canonical projective submanifold of complex type of complex dimension $k$}. In this case, we will denote $W$ as $\P\C^{k}$.

\medskip

If  $V^{k+1} \subseteq  \C^{n+1}(a)$  is indefinite, then the intersection
of its projectivization with ${\rm H^{n}_{\Q}}$ is a totally geodesic submanifold of ${\rm H^{n}_{\Q}}$. We call this  submanifold of ${\rm H^{n}_{\Q}}$ a {\sl totally geodesic submanifold of complex type of complex dimension $k$}. When $V^{k+1} \subseteq  \C^{n+1}$, we call this totally geodesic submanifold a {\sl canonical totally geodesic  submanifold of complex type of complex dimension $k$}, or a {\sl canonical complex hyperbolic submanifold of dimension $k$} of ${\rm H^{n}_{\Q}}$.  In this case, we will denote this submanifold as ${\rm H^{k}_{\C}}$.

\subsubsection{Stabilizators of totally geodesic submanifolds}\label{stab}

\noindent

Let $M$ be a totally geodesic submanifold in ${\rm H^{n}_{\F}}$. Let $I(M)$ denote the subgroup of ${\rm PU}(n,1; \F)$  which leaves $M$ invariant.

\medskip

If $F$ is a subfield of $\F$, we let  ${\mathcal{N}}^{+}(F,\F)=\{\lambda\in\F^{+};\lambda F\lambda^{-1}=F\}$, where $\F^{+}$ denotes the subgroup in $\F$ of elements with norm one.

\medskip

The following propositions can be found in \cite{CGr}, p.74.

\begin{prop} 
Let $M$ be a totally geodesic submanifold, such that $M$ is the intersection of the projectivization of an $F$-hyperbolic subspace of $\F^{n,1}$ with ${\rm H^{n}_{\F}}$ for some
subfield $F$ of $\F$. Then the elements $g\in I(M)$ are of the form:

$$g=\left(\begin{array}{rr}A\lambda&0\\0&B \end{array}\right),$$
where $A\in U(m,1;F)$, $ \lambda\in{\mathcal{N}}^{+}(F,\F),$ and $B\in U(n-m;\F).$
\end{prop}

\medskip

\begin{prop}
\label{stab1} 
Let $M$ be a 3-dimensional  totally geodesic submanifold of a totally geodesic quaternionic line  ${\rm H^{1}_{\Q}}$  in ${\rm H^{n}_{\Q}}$. Then the elements $g\in I(M)$ are of the form:

$$g=\left(\begin{array}{rr}A&0\\0&B \end{array}\right), \text{ where } A=\left(\begin{array}{rr}a&b\\ -\varepsilon b & \varepsilon a \end{array}\right)\in U(1,1;\Q),$$ \\
$\varepsilon = \pm 1$ and $B\in U(n-1;\Q)$.
\end{prop}

\subsubsection{A little more about the isometry group of the quaternionic hyperbolic space}\label{lmaig}

\noindent Let us consider the complex hyperbolic space ${\rm H^{n}_{\C}} $. It has a natural complex structure related to its isometry group, and the isometry group of ${\rm H^{n}_{\Q}}$
is generated by the holomorphic isometry group, which is the projective group ${\rm PU}(n,1; \C)$, and the anti-holomorphic isometry $\sigma$ induced by complex
conjugation in $\C ^{n+1}$. This anti-holomorphic isometry corresponds to the unique non-trivial automorphism of the field of complex numbers. Below we consider a similar isometry
of quaternionic hyperbolic space ${\rm H^{n}_{\Q}}$.

\medskip
We recall that if $f: \Q \rightarrow \Q$ is an automorphism of $\Q$, then $f$ is an inner automorphism of $\Q$, that is,
$f(q)= a q a^{-1}$ for some $a \in \Q, a \neq 0$.

\medskip

It follows from the fundamental theorem of projective geometry, see \cite{Art}, that each projective map $L: \P\Q^{n}\rightarrow \P\Q^{n}$ is induced by a semilinear or linear map $\tilde{L}: \Q^{n+1}\rightarrow \Q^{n+1}$.

\medskip

It is easy to see that if a projective map

$$L: \P\Q^{n}\rightarrow \P\Q^{n}$$

is induced by a semilinear map

$$\tilde{L}: \Q^{n+1}\rightarrow \Q^{n+1}, \tilde{L}(v)=a v a^{-1},v \in \Q^{n+1}, a\in \Q \setminus \R,$$

then it is  also induced by a linear map $v \mapsto av$.

\medskip

Therefore, the projective group of
$\P\Q^{n}$  is the projectivization of the linear group of  $\Q^{n+1}$.

\medskip

This implies that if

$$L: {\rm H^{n}_{\Q}} \rightarrow {\rm H^{n}_{\Q}} $$

is an isometry, then $L$ is induced  by a linear isometry

$$\tilde{L}: \Q^{n,1} \rightarrow \Q^{n,1}.$$

\medskip

This explains why the group of all isometries of ${\rm H^{n}_{\Q}}$
is the projectivization of the linear group  ${\rm U}(n,1; \Q)$, that is, ${\rm PU}(n,1; \Q).$

\medskip

Next we consider a curious map, which is an isometry of the quaternionic hyperbolic space, that has no analogue in geometries over
commutative fields. Let $\tilde{L}_a : v \mapsto av$, $v \in \Q^{n+1}$,  $a\in\Q$, $a$ is not real. The projectivization of this linear map
defines a non-trivial map $L_a:\P\Q^{n}\rightarrow \P\Q^{n}$. We remark that in projective spaces over commutative fields this map $L_a$ is identity.
It easy to see that $\tilde{L}_a \in {\rm U}(n,1; \Q)$ if and only if $|a|=1$, so $L_a$ is in ${\rm PU}(n,1; \Q)$ if and only if $|a|=1$.

\begin{prop}
\label{fixp}
 Let $a \in \Q$ be such that $|a|=1$. Then the fixed point set $S_a$ of $L_a$  is a totally geodesic submanifold of complex type of maximal dimension in ${\rm H^{n}_{\Q}}.$
This submanifold is globally equivalent to the canonical complex hyperbolic submanifold ${\rm H^{n}_{\C}}$ of ${\rm H^{n}_{\Q}}$.
\end{prop}

\begin{proof} The proof follows from Proposition \ref{qua3}.
\end{proof}

\medskip

It is also easy to see that if $a$ is purely imaginary and $|a|=1$, then  $L_a$ is an involution. We call this isometry $L_a$ a {\sl geodesic reflection} with respect to $S_a$, or a geodesic reflection in $S_a$.

\section{Bisectors in quaternionic hyperbolic space}

\noindent

In ${\rm H^{n}_{\F}}$, as it follows from Proposition \ref{ptgsm}, totally geodesic (real) hypersurfaces exist only when $\F=\R$.
Therefore, when $\F \neq  \R$, a reasonable substitute  are the bisectors which are close as possible to being totally geodesic.
The reader can find a comprehensive study of bisectors in complex hyperbolic space in the Goldman book \cite{Gol}.

\medskip

Giraud \cite{Gir} and Mostow \cite{Mos} described the structure of bisectors in complex hyperbolic space ${\rm H^{n}_{\C}}$ in terms of a foliation by
totally geodesic complex submanifolds of maximal dimension. Goldman \cite{Gol} proved that bisectors enjoy another decomposition
into totally real, totally geodesic submanifolds, which is called the {\sl meridian decomposition}. It is important that
these decompositions are unique.

\medskip

In this section, we will describe various decompositions of bisectors in quaternionic hyperbolic space.
We show that geometry of bisectors in quaternionic hyperbolic space is more rich than in complex hyperbolic geometry. First, we prove an analogue
of the Mostow decomposition of bisectors in ${\rm H^{n}_{\Q}}$.
Finally, we will show that any bisector in quaternionic hyperbolic space ${\rm H^{n}_{\Q}}$ is the union of totally geodesic submanifolds of ${\rm H^{n}_{\Q}}$
isometric to ${\rm H^{n}_{\C}}$ intersecting in a common point. In some sense, this is an analogue of the Goldman meridian decomposition. We call such decompositions
the {\sl fan decompositions}. The existence of fan decompositions implies, in particular,
that any bisector in quaternionic hyperbolic space is star-like with respect to any point in its real spine.

The proof of these results is based on the classification of reflective submanifolds in quaternionic hyperbolic space \cite{Leu}.

\subsection{Bisectors in hyperbolic spaces}

Let $p_1, p_2$ be two distinct points in ${\rm H^{n}_{\F}}$. The {\sl bisector equidistant from $p_1$ and $p_2$,} or the {\sl bisector of }$\{p_1,p_2 \}$ is defined as

$$B(p_1,p_2)=\{p \in {\rm H^{n}_{\F}}: d(p_1, p)=d(p_2,p)\}.$$

\medskip

An equidistant hypersurface or  {\sl bisector} in  ${\rm H^{n}_{\F}}$ is a subset $B=B(p_1,p_2)$ for some pair of points $p_1, p_2$. 

\medskip

In this case, we will say that $B$
is equidistant from $p_1$ (or equidistant from $p_2$).

\medskip

Next we consider the following two cases: (1) $\F=\C$, and (2) $\F=\Q$.

\subsubsection{Bisectors in complex hyperbolic space}

In this section, we recall some basic results on bisectors in complex hyperbolic space.
\medskip

Let $\{p_1,p_2\}$ be as above. Then there exists the unique complex geodesic $\Sigma \subset {\rm H^{n}_{\C}}$ spanned by $p_1$ and $p_2$. Following Mostow \cite{Mos},
we call $\Sigma$ the {\sl complex spine}  with respect to the pair $\{p_1,p_2\}$. The {\sl spine} of $B$ (with respect to $\{p_1,p_2\}$)
equals

$$\sigma \{p_1,p_2\}= B(p_1,p_2) \cap \Sigma = \{ p \in \Sigma : d(p_1,p)=d(p_2,p)\}$$
that is, the orthogonal bisector of the geodesic segment joining $p_1$ and $p_2$ in $\Sigma$.

\medskip

The following result is due to Giraud \cite{Gir} and Mostow \cite{Mos}, see also \cite{Gol}.

\begin{theorem}

Let $B, \Sigma$ and $\sigma$ be as above. Let $\Pi_\Sigma :  {\rm H^{n}_{\C}} \rightarrow \Sigma$ be the orthogonal projection onto $\Sigma$. Then

$$B = \Pi_\Sigma ^{-1} (\sigma)=\bigcup_{s \in \sigma} \Pi_\Sigma ^{-1}(s)$$.

\end{theorem}

\medskip

The complex hyperplanes $\Pi_\Sigma ^{-1}(s)$, for $s \in \sigma$, are called the {\sl slices} of $B$ (with respect to $\{p_1,p_2\}$), and the decomposition
of a bisector into slices is called the {\sl slice decomposition}, or the {\sl Mostow decomposition}.

\medskip

Since orthogonal projection $\Pi_\Sigma :  {\rm H^{n}_{\C}} \rightarrow \Sigma$ is a real analytic fibration, we obtain that a bisector is a real
analytic hypersurface in ${\rm H^{n}_{\C}}$ diffeomorphic to $\R^{2n-1}$.

\medskip
It was shown by Goldman  \cite{Gol} that the slices, the spine and complex spine of a bisector $B$ depend intrinsically on the hypersurface $B$, and not on the defining pair $\{p_1,p_2\}$.

\begin{theorem}
Suppose that $\{p_1,p_2\}$ and  $\{p^\prime_1,p^\prime_2\}$ are two pairs of distinct points in complex hyperbolic space ${\rm H^{n}_{\C}}$ such that $B=B(p_1,p_2)$ and $B^\prime=B(p^\prime_1,p^\prime_2)$
are equal. Then the slices (respectively spine, complex spine) of $B$ with respect to $\{p_1,p_2\}$ equal the slices (respectively spine, complex spine) of $B^\prime$
with respect to $\{p^\prime_1,p^\prime_2\}$.
\end{theorem}

\medskip

Goldman's proof of this theorem is based on the following: the complex hyperbolic space is a complex analytic manifold; a bisector is Levi-flat and its maximal holomorphic
submanifolds are its slices.

\medskip

The endpoints of the spine of a bisector $B$ are called the {\sl vertices} of $B$. Since a geodesic $\sigma \in {\rm H^{n}_{\C}}$ is completely determined by the unordered
pair $\partial \sigma \subset \partial {\rm H^{n}_{\C}}$ consisting of its endpoints, bisectors are completely parametrized by unordered pairs of distinct points in $\partial {\rm H^{n}_{\C}}$.
Therefore, there is a duality between bisectors and geodesics. We associate to every bisector a geodesic, its spine $\sigma$. Conversely, if $\sigma \subset {\rm H^{n}_{\C}}$ is
a geodesic, then there exists a unique bisector $B=B_\sigma$ with the spine $\sigma$. Given a geodesic $\sigma$, there exists a unique complex geodesic $\Sigma \supset \sigma$.
Let $R_\sigma :\Sigma\rightarrow \Sigma$ be the unique reflection whose fixed-point set is $\sigma$. Take an arbitrary point $p_1 \in \Sigma \smallsetminus \sigma$
and let $p_2=R_\sigma (p_1)$. Then

$$B=B(p_1,p_2)=\Pi_\Sigma ^{-1}(\sigma)$$
is a bisector having $\sigma$ as its spine. Therefore, we have

\begin{theorem}
There is a natural bijective correspondence between bisectors in ${\rm H^{n}_{\C}}$ and geodesics in ${\rm H^{n}_{\C}}$.
\end{theorem}

\medskip

Next, we recall so called the {\sl meridianal decomposition} of bisectors in complex hyperbolic space invented  by Goldman \cite{Gol}.
\medskip

The following result shows that bisectors decompose into totally real geodesic submanifolds as well as into complex totally geodesic submanifolds of complex dimension $n-1$.

\begin{theorem}\label{bcomp1}
Let $\sigma \in {\rm H^{n}_{\C}}$ be a real geodesic. For each $k\geq 2$, the bisector $B$ having spine $\sigma$ is the union of all $\R^k$-planes containing $\sigma$.
\end{theorem}

\medskip

Following Goldman, we refer to the $\R^n$-planes containing the spine as the {\sl meridians} of the bisector $B$.

\medskip

We remark that in contrast to the Mostow decomposition, which is a non-singular foliation, the meridianal decomposition is a singular foliation having the geodesic $\sigma$ as its singular set.

\medskip

Goldman's proof of Theorem \ref{bcomp1} is based on a sufficiently complicate computation. As a corollary of the existence of the meridianal decomposition, we have the following.

\begin{theorem} Bisectors in complex hyperbolic space are star-like with respect to any point in its real spine.
\end{theorem}

\subsection{Bisectors in quaternionic hyperbolic space}

In this section, we will study bisectors in $n$- dimensional quaternionic hyperbolic space ${\rm H^{n}_{\Q}}$.

\medskip

As we have seen in the previous section, the orthogonal projections play an
important role in the study of bisectors. The first problem we consider is when the orthogonal projections onto totally geodesic submanifolds have totally geodesic fibers.

\subsubsection{Orthogonal projections and reflective submanifolds}

It is elementary that in real hyperbolic space ${\rm H^{n}_{\R}}$, the orthogonal projections onto totally geodesic submanifolds have totally geodesic fibers.
This is not true, in general, in hyperbolic spaces with non-constant sectional curvature.

\medskip

For instance, if $M$ is a totally geodesic submanifold of real dimension one, that is, a real geodesic,
in complex hyperbolic space ${\rm H^{2}_{\C}}$, then the fibers of the orthogonal projection onto $M$ are not totally geodesic, because in this case the fibers have real dimension $3$,
but we know that any proper totally geodesic submanifold in ${\rm H^{2}_{\C}}$ has real dimension $1$ or $2$.

\medskip

Another example: let $M$ be a totally real totally geodesic plane in ${\rm H^{3}_{\C}}$ , that is,
$M$ is isometric to ${\rm H^{2}_{\R}}$. Then, using again the classification of totally geodesic submanifolds in ${\rm H^{3}_{\C}}$, it is easy to prove that the fibers
of the orthogonal projection onto $M$ are not totally geodesic.

\medskip

One more example: let $M$ be a totally geodesic submanifold in quaternionic hyperbolic space ${\rm H^{2}_{\Q}}$ of complex type of real dimension $2$, that is, $M$ is isometric
to ${\rm H^{1}_{\C}}$. Then it can be shown that the fibers of the orthogonal projection onto $M$ are not totally geodesic.

\medskip

We will show that this problem is related to so-called reflective submanifolds of hyperbolic spaces.

\medskip

First, we recall the definition of reflective submanifolds.

\medskip

Let $N$ be a Riemannian manifold and $M$ a submanifold of $N$. Then $M$ is called {\sl reflective} if the
geodesic reflection of $N$ with respect to $M$ is a globally well-defined isometry of $N$.

\medskip

Since any reflective submanifold is a connected component of  the fixed point set of an isometry, it is totally geodesic.

\medskip

Also, we recall the definition of symmetric submanifolds.

\medskip
A submanifold $M$ of a Riemannian manifold $N$ is called {\sl symmetric} if for each point $p$ in $M$ there exists an isometry $I_p$ of $N$ such that

$$ I_p(p)=p,\ \ \  I_p(M)=M,\ \ \  (I_p)_* X =-X,\ \ \ (I_p)_* Y=Y$$
for all $X \in T_p M, \ Y \in \nu_p M.$

\medskip

Here we denote by $T_p M$ the tangent space of $M$ at $p$,  by $\nu_p M$ the normal space of $M$ at $p$, and by $(I_p)_*$ the differential of $I_p$.

\medskip

For symmetric spaces there is the following useful criterion.

\begin{prop} A totally geodesic submanifold of a simply connected Riemannian symmetric space is symmetric if and only if it is reflective.
\end{prop}

\medskip

As a reflective submanifold is symmetric, at each point there exists a complementary totally geodesic submanifold normal to it. In symmetric spaces this normal submanifold
is also reflective. Since all hyperbolic spaces are symmetric, we have that this result holds for any hyperbolic space.

\medskip

Let $M$ be a reflective submanifold in a symmetric simply connected Riemannian space $N$, and $p \in M$. Let $M_p^\bot$ denote its {\sl orthogonal complement} at $p$, that is, the connected,
complete, totally geodesic submanifold of $N$ with $TM_p^\bot =\nu_p M$.

\medskip

The following theorems describe all reflective submanifolds in complex and quaternionic hyperbolic spaces, see \cite{Leu}.

\medskip

\begin{theorem}
Let $M$ be a reflective submanifold in complex hyperbolic space ${\rm H^{n}_{\C}}$. Then $M$ is either isometric to ${\rm H^{k}_{\C}}$, $k=1,\ldots, n-1$, or to ${\rm H^{n}_{\R}}$.
\end{theorem}

\begin{theorem} 
Let $M$ be a reflective submanifold in quaternionic  hyperbolic space ${\rm H^{n}_{\Q}}$. Then $M$ is either isometric  to ${\rm H^{k}_{\Q}}$, $k=1,\ldots, n-1$, or to ${\rm H^{n}_{\C}}$.
\end{theorem}

\begin{co} 
Let $M$ be a reflective submanifold in complex hyperbolic space ${\rm H^{n}_{\C}}$, and $p \in M$.  Then $M_p^\bot$ is reflective, and it is isometric to ${\rm H^{n-k}_{\C}}$ if
$M$ is isometric to ${\rm H^{k}_{\C}},$ and $M_p^\bot$ is isometric to ${\rm H^{n}_{\R}}$ if $M$ is isometric to ${\rm H^{n}_{\R}}$.
\end{co}

\begin{co} 
Let $M$ be a reflective submanifold in quaternionic hyperbolic space ${\rm H^{n}_{\Q}}$. Then $M_p^\bot$ is reflective, and it is isometric to ${\rm H^{n-k}_{\Q}}$ if
$M$ is isometric to ${\rm H^{k}_{\Q}}$, and $M_p^\bot$  is isometric to ${\rm H^{n}_{\C}}$ if $M$ is isometric to ${\rm H^{n}_{\C}}$.
\end{co}

All this implies the following results.

\begin{theorem} 
Let $M$ be a totally geodesic submanifold in complex hyperbolic space ${\rm H^{n}_{\C}}$. Then the fibers of the orthogonal projection onto $M$ are
totally geodesic if and only if $M$ is either isometric to ${\rm H^{k}_{\C}}$, $k=1,\ldots, n-1$, or to ${\rm H^{n}_{\R}}$.
\end{theorem}

\begin{theorem} 
\label{refl1} Let $M$ be a totally geodesic submanifold in quaternionic hyperbolic space ${\rm H^{n}_{\Q}}$. Then the fibers of the orthogonal projection onto $M$ are
totally geodesic if and only if $M$  is either isometric to ${\rm H^{k}_{\Q}}$, $k=1,\ldots, n-1$, or to ${\rm H^{n}_{\C}}$.
\end{theorem}

\subsubsection{Mostow decomposition of bisectors in quaternionic hyperbolic space}

In this section, we describe the Mostow decomposition of bisectors in  quaternionic hyperbolic space ${\rm H^{n}_{\Q}}$.

\medskip

Let $p_1, p_2$ be two distinct points in ${\rm H^{n}_{\Q}}$, and let $B(p_1,p_2)$ be the bisector of $\{p_1,p_2\}$, that is,

$$B(p_1,p_2)=\{p \in {\rm H^{n}_{\Q}}: d(p_1, p)=d(p_2,p)\}.$$

Then it follows from projective geometry that there exists the unique quaternionic geodesic $\Sigma \subset {\rm H^{n}_{\Q}}$ spanned by $p_1$ and $p_2$.
We call $\Sigma$ the {\sl quaternionic spine}  with respect to the pair $\{p_1,p_2\}$. The {\sl real spine} of $B$ (with respect to $\{p_1,p_2\}$)
equals

$$\sigma \{p_1,p_2\}= B(p_1,p_2) \cap \Sigma = \{ p \in \Sigma : d(p_1,p)=d(p_2,p)\}$$
that is, the orthogonal bisector of the geodesic segment joining $p_1$ and $p_2$ in $\Sigma$.
Since $\Sigma$ is isometric to ${\rm H^{4}_{\R}}$, it follows that $\sigma$ is isometric to ${\rm H^{3}_{\R}}$.

\medskip

Let $\Pi_\Sigma :  {\rm H^{n}_{\Q}} \rightarrow \Sigma$ be orthogonal projection onto $\Sigma$. It follows from Theorem \ref{refl1} that the fibers of $\Pi$ are totally
geodesic submanifolds of ${\rm H^{n}_{\Q}}$ isometric to  ${\rm H^{n-1}_{\Q}}$.

\medskip

First,  we prove the following proposition.

\begin{prop} 
\label{dmos1} For any $p \in {\rm H^{n}_{\Q}} \setminus \Sigma$, and $q \in \Sigma$, the geodesics from $\Pi_\Sigma (p)$ to $p$ and $q$  are orthogonal  and span a
totally real totally geodesic $2$-plane.
\end{prop}

\begin{proof} 
To prove this, we will use the ball model for  ${\rm H^{n}_{\Q}}$, see Section \ref{ball}. In what follows, we identify ${\rm H^{n}_{\Q}}$
with the unit ball  $\D$ in $\Q^n$,

$$\D=\{ (z_1,z_2, \ldots, z_n) \in \Q^n :  |z_1|^2 + |z_2|^2 + \ldots + |z_n|^2 < 1   \}.$$

We write $(z_1,z_2, \ldots, z_n)=(z_1, z^\prime)$, where $z^\prime = (z_2, \ldots , z_n) \in  \Q^{n-1}.$ If $z_2= z_3= \ldots =z_n =0$, we denote $z^\prime$ by $0^\prime $.

\medskip

Since the isometry group of ${\rm H^{n}_{\Q}}$ acts transitively on the set of quaternionic geodesics in ${\rm H^{n}_{\Q}}$, we can assume without loss of generality that
$\Sigma= \{(z, 0^\prime) \in \D : z \in \Q \}$. Therefore, in this case, if $p=(z, z^\prime) \in \D$, we have that   $\Pi_\Sigma (p)=(z,0^\prime)$.

\medskip

Let $p=(z,z^\prime) \in \D \setminus \Sigma $, $p^\prime= \Pi_\Sigma (p)= (z, 0^\prime)$, and $q=(w, 0^\prime) \in \Sigma , w \in \Q$.
Let $P=(z,z^\prime , 1)^t$, $P^\prime=(z, 0^\prime, 1)^t$, $Q=(w, 0^\prime,1)^t$ be
vectors in $\Q^{n,1}$ representing the points $p, p^\prime , q$ respectively.

\medskip

Now we compute the Hermitian triple product $\herm{P,P^\prime,Q}$. We have

$$\herm{P,P^\prime, Q} =\herm{P,P^\prime }\herm{P^\prime,Q}\herm{Q,P} = \\ (\bar{z}z - 1)(\bar{z}w -1)(\bar{w}z -1)= \\
(|z|^2 -1)|\bar{z} w -1|^2.$$

So, the Hermitian triple product $\herm{P,P^\prime,Q}$ is real. Therefore, the points $p, p^\prime , q$ lie in a totally real geodesic plane in ${\rm H^{n}_{\Q}}$.
\end{proof}

\medskip

Also we need so-called the Pythagorean theorem in hyperbolic plane ${\rm H^{2}_{\R}}$, see \cite{Ber}.

\begin{theorem} Let $a, b, c$ be vertices of a right triangle with right angle at $a$. Then

$$\cosh d(b,c)=\cosh d(a,b) \cosh d(a,c).$$

\end{theorem}

\medskip

Now we are ready to prove the Mostow decomposition theorem in quaternionic hyperbolic geometry.

\begin{theorem}\label{qmos}

Let $B, \Sigma$ and $\sigma$ be as above. Let $\Pi_\Sigma :  {\rm H^{n}_{\Q}} \rightarrow \Sigma$ be orthogonal projection onto $\Sigma$. Then

$$B = \Pi_\Sigma ^{-1} (\sigma)=\bigcup_{s \in \sigma} \Pi_\Sigma ^{-1}(s)$$.

\end{theorem}

\begin{proof} 
Let $p \in \Pi_\Sigma ^{-1} (\sigma)$.
First, we have that $d(\Pi_\Sigma(p),p_1)=d(\Pi_\Sigma,p_2)$. By applying Proposition \ref{dmos1}, we have that
the points $p,\  \Pi_\Sigma(p),\  p_1$ lie in a totally real geodesic plane in ${\rm H^{n}_{\Q}}$. The same is true for the points $p,\  \Pi_\Sigma(p),\  p_2.$
Then it follows from the Pythagorean theorem applied to the right triangles with vertices $p,\  \Pi_\Sigma(p),\  p_1$  and $p,\  \Pi_\Sigma(p),\  p_2$
that $d(p_1,p)=d(p_2,p)$, that is, $p \in B.$

\end{proof}

\medskip

We call the quaternionic hyperplanes  $\Pi_\Sigma ^{-1} (p)$, for $p \in \sigma$, the {\sl slices} of the bisector $B$ (with respect to $\{p_1, p_2\}$).
It is clear that any two distinct slices of $B$ are ultra-parallel, that is, the distance  between them is greater then zero.
\medskip

Since orthogonal projection $\Pi_\Sigma :  {\rm H^{n}_{\Q}} \rightarrow \Sigma$ is a real analytic fibration, we have:

\begin{co} 
A quaternionic  bisector is a real analytic hypersurface in ${\rm H^{n}_{\Q}}$ diffeomorphic to $\R^{4n -1}$.
\end{co}

\medskip

Now we show that the quaternionic spine $\Sigma$ and the real spine $\sigma$ of $B$ depend intrinsically on $B$, and not on the pair $\{p_1, p_2\}$ used to define $B$.
As was remarked above, Goldman's proof of the complex hyperbolic analogue of this fact is based on the following: the complex hyperbolic space is a complex analytic manifold;
a bisector is Levi-flat and its maximal holomorphic submanifolds are its slices. Since quaternionic hyperbolic space has no natural complex structure, this proof does not
work in the quaternioic case. We prove this result using some elementary facts in quaternionic projective geometry.

\medskip

\begin{theorem} 
Let us suppose that $\{p_1, p_2\}$ and $\{p^\prime _1, p^ \prime _2\}$ be two pairs of distinct points in  ${\rm H^{n}_{\Q}}$
such that the bisectors $B=B(p_1, p_2)$ and $B^\prime = B^\prime (p^\prime _1, p^\prime _2)$ are equal. Then the slices (respectively, quaternioic spine, real spine) of $B$
with respect to  $\{p_1, p_2\}$ equal the slices (respectively, quaternioic spine, real spine) of $B^\prime$ with respect to $\{p^\prime_1, p^\prime _2\}$.
\end{theorem}

\begin{proof} First we show that the slices of $B$ coincide with the slices of $B^\prime$.

\medskip

We need the following facts from projective geometry, see \cite{Art}: any two different projective submanifolds in quaternionic projective space $\P\Q^n$, $n>1$, of quaternionic dimension $n-1$ are transversal along their intersection; every intersection of projective
submanifolds in $\P\Q^n$ is a projective submanifold in $\P\Q^n$. Also, we recall that any totally geodesic submanifold of ${\rm H^{n}_{\Q}}$
isometric to ${\rm H^{k}_{\Q}}$, $k=1, \ldots, n-1$, is the intersection of a projective submanifold in $\P\Q^n$ of dimension $k$ with ${\rm H^{n}_{\Q}}$.

\medskip

Let $s^\prime$ be a slice of $B^\prime$. Since $B=B^\prime$, then $s^\prime$ is either a slice of $B$, or $s^\prime$ intersects a slice $s$ of $B$ transversally.
Suppose that $s \cap s^\prime \neq \emptyset$ and $s \neq s^\prime$. Since $s \subset B$ and $s^\prime \subset B$, and $B$ is a smooth submanifold of ${\rm H^{n}_{\Q}}$
of real dimension $4n-1$, it follows from transversality of $s$ and $s^\prime$ along their intersection that $ \dim_\R s + \dim_\R s^\prime - (4n-1) = \dim_\R (s\cap s^\prime)$.
This implies that $\dim_\R (s\cap s^\prime)=4n-7$. This is impossible, since $s \cap s^\prime$ is the intersection of a projective submanifold in $\P\Q^n$ with ${\rm H^{n}_{\Q}}$.
Therefore, we get that $s=s^\prime$. So, the slices of $B^\prime$ are slices of $B$. Since each slice is orthogonal to $\Sigma$, any pair of distinct slices of bisector are
ultraparallel and their unique common orthogonal quaternionic geodesic is equal to $\Sigma$.

\medskip

Therefore, $B$ completely determines the quaternionic spine $\Sigma$. Since $\sigma= B \cap \Sigma$, the bisector $B$ uniquely determines its real spine.
\end{proof}

\medskip

The following theorem is a generalization of the Mostow decomposition theorem.

\begin{theorem} Let $\Sigma$ be a totally geodesic submanifold of ${\rm H^{n}_{\Q}}$ isometric to ${\rm H^k_{\Q}}$, $k=1$, $2,\ldots ,n-1$. Let $\Pi_\Sigma   :{\rm H^n_{\Q}}\to\Sigma$ be the orthogonal projection onto $\Sigma$. If $B'\subset\Sigma$ is a bisector in $\Sigma,$ $B'= B'(p_1,p_2),$
$p_1,p_2 \in  \Sigma$, then $B=\Pi_\Sigma ^{-1}(B')$ is a bisector in ${\rm H^n_{\Q}}$ such that $B'=B\cap\Sigma$, and, moreover, $B=B(p_1,p_2)$. Conversely, if $B=B(p_1,p_2)$ is a bisector in ${\rm H^n_{\Q}}$ defined by $p_1,p_2 \in  \Sigma$, then its orthogonal projection $B'$ onto $\Sigma$ is a bisector in $\Sigma$, and  $B'= B'(p_1,p_2).$
\end{theorem}
\begin{proof} The proof follows along the lines of the proof of Theorem 2.11 and the lemma below.
\end{proof}

\medskip

\begin{lemma}
Let $\Sigma$ be a totally geodesic submanifold of ${\rm H^n_{\Q}}$ isometric to ${\rm H^k_{\Q}}$, $k=1, \ldots ,n-1$. Let $\Pi_\Sigma: {\rm H^n_{\Q}}\to\Sigma$ be the orthogonal projection onto $\Sigma$. If $p\in {\rm H^n_{\Q}} \setminus \Sigma$ and $q\in\Sigma$, then the points $\Pi_\Sigma (p)$, $p$ and $q$ span a totally real geodesic 2-plane in ${\rm H^n_{\Q}}$. 
\end{lemma}
\begin{proof}
Let us consider the ball model $\D$ for ${\rm H^n_{\Q}}$. For $p\in\D$, we write $p=(z,z')$, where $z\in\Q^k$ and $z'\in\Q^{n-k}$. Applying an automorphism, we can assume that $\Sigma=\{(z,0')\in\D:\,z\in\Q^k\}$. So, $\Pi_\Sigma (z,z')=(z,0')$. 

\medskip

Now, let $p=(z,z')\in\D \setminus \Sigma$, $p'=(z,0')=\Pi_\Sigma (p)$ and $q=(w,0')\in\Sigma$. Let $P=(z,z',1)^t$, $P'=(z,0',1)^t$, and $Q=(w,0',1)^t$ be vectors in $\Q^{n,1}$ representing the points $p$, $p'$, $q$, respectively. We have that
\begin{align*}
\herm{P,P',Q}&=\herm{P,P'}\herm{P',Q}\herm{Q,P}=\\
             &(\hherm{z,z}-1)(\hherm{z,w}-1)(\hherm{w,z}-1)=\\
             &(\hherm{z,z}-1)~|\hherm{z,w}-1|^2\in\R.
\end{align*}

Therefore, $p, p', q$ lie in a totally real geodesic plane in ${\rm H^n_{\Q}}$. 
\end{proof}

\subsubsection{Automorphism of bisectors}

If $N$ is a totally geodesic submanifold of ${\rm H^{n}_{\Q}}$ isometric to ${\rm H^{3}_{\R}}$, then there exists the unique quaternionic geodesic
$M$ which contains $N$. Then the pair $\{M, N\}$ defines a bisector $B$ whose quaternionic spine equals to $M$ and real spine equals to $N$.
Indeed, take a point $p$ in $M \setminus N$. Let $\gamma$ be the unique real geodesic in $M$ which contains $p$ and which is orthogonal to $N$. Let $p^\prime \in \gamma$ be
the point symmetric to $p$. It is clear that $M$ is the quaternionic spine of the bisector $B=B(p, p^\prime)$, and $N$ is the real spine  of $B$.

\medskip

Since the group of isometries $PU(n, 1, \Q)$ of ${\rm H^{n}_{\Q}}$ acts transitively on such pairs $\{M, N\}$, we get that $PU(n, 1, \Q)$ acts transitively on bisectors.
Since the quaternionic geodesic $M$ containing $N$ is unique, we have that a bisector in ${\rm H^{n}_{\Q}}$ is defined uniquely by its real spine. Furthermore, the stabilizer of a bisector
in $PU(n, 1, \Q)$ equals the stabilizer of its real spine. This group is described in Proposition \ref{stab1}.

\medskip

\subsubsection{Orthogonality of totally geodesic submanifolds of complex type in  quaternionic hyperbolic space}

Suppose that $M$ and $N$ are totally geodesic submanifolds in  ${\rm H^{n}_{\Q}}$ isometric to ${\rm H^{n}_{\C}}$.

\medskip

Let $I_M$ and $I_N$ denote the geodesic reflections of ${\rm H^{n}_{\Q}}$ with respect to $M$ and $N$ respectively.
Since  $M$ and $N$ are reflective, $I_M$ and $I_N$ are isometries of ${\rm H^{n}_{\Q}}$. The fixed point set of $I_M$ equals to $M$,
and the fixed point set of $I_N$ equals to $N$.

\medskip

The proof of the following theorem is standard.

\medskip

\begin{theorem} The following conditions are equivalent:

\begin{enumerate}
\item $I_M$ and $I_N$ commute,
\item $I_M (N)=N$,
\item $I_N (M)=M$,
\item $M$ and $N$ intersect orthogonally in ${\rm H^{n}_{\Q}}$.
\end{enumerate}

\end{theorem}

\medskip

We recall that ${\rm H^{n}_{\C}}$ has a totally geodesic submanifold isometric to  ${\rm H^{n}_{\R}}$.

\begin{theorem}
\label{ortg1}
Let $M$ be a totally geodesic submanifold of  ${\rm H^{n}_{\Q}}$ isometric to ${\rm H^{n}_{\C}}$. Then for any totally geodesic
submanifold $S \subset M$ isometric to  ${\rm H^{n}_{\R}}$ there exists a  totally geodesic submanifold $N$ isometric to ${\rm H^{n}_{\C}}$ such that
$ S=M \cap N$ and $N$ is orthogonal to $M$ along $S$.
\end{theorem}

\begin{proof} 
To prove this, we will use the ball model for  ${\rm H^{n}_{\Q}}$, see Section \ref{ball}. In what follows, we identify ${\rm H^{n}_{\Q}}$
with the unit ball  $\D$ in $\Q^n$,

$$\D=\{ (q_1,q_2, \ldots, q_n) \in \Q^n :  |q_1|^2 + |q_2|^2 + \ldots + |q_n|^2 < 1   \}.$$

\medskip

Let us consider the following subsets of $\D$:

$$\D_\C=\{ (z_1,z_2, \ldots, z_n) \in \C^n :  |z_1|^2 + |z_2|^2 + \ldots + |z_n|^2 < 1   \},$$

$$\D_\R=\{ (x_1,x_2, \ldots, x_n) \in \R^n :  |x_1|^2 + |x_2|^2 + \ldots + |x_n|^2 < 1   \},$$
where $\C = \{z=x + iy : x, y \in \R \}$ and $\R$ is the standard subfield of real numbers of $\Q$.

\medskip

Then $\D_\C$ is a totally geodesic submanifold of $\D$ isometric to ${\rm H^{n}_{\C}}$, and $\D_\R$ is a totally geodesic submanifold of  $\D_\C$
isometric to ${\rm H^{n}_{\R}}$.

\medskip

Since the group of isometries of ${\rm H^{n}_{\Q}}$ acts transitively on the set of pairs of the form $(M,S)$ in  ${\rm H^{n}_{\Q}}$,
we can assume without loss of generality that $M=\D_\C$ and $S=\D_\R$.

\medskip

Let $a \in \Q$ be purely imaginary, $|a|=1$. Then $a^2=-1$. Let $\C(a)$ be the subfield of $\Q$ spanned by $1$ and $a$. Denote by $\D_a$ the following set in $\D$ :

$$\D_a=\{ (a_1,a_2, \ldots, a_n) \in \C(a)^n :  |a_1|^2 + |a_2|^2 + \ldots + |a_n|^2 < 1  \}.$$

\medskip

We know, see Section \ref{tgsm}, that $\D_a$ is a totally geodesic submanifold of $\D$ isometric to ${\rm H^{n}_{\C}}$. We remark that $\D_a = \D_\C$ if $a=i$.
It is clear that $\D_\C \cap \D_a = \D_\R $, provided that  $a \neq i$.

\medskip

Let $V_\C$, $V_a$, $V_\R$ denote the sets of vectors in $\Q^{n,1}$ with coordinates in $\C$, $\C(a)$, and $\R$ respectively. Then we have that
$\D_\C$,  $\D_a$,  $\D_\R$ are  projectivizations of $V_\C$, $V_a$, and $V_\R$ respectively.

\medskip

Consider the following semi-linear map $L_a: \Q^{n,1}  \longrightarrow \Q^{n,1}$ given by $L_a (v)=ava^{-1}$, $v \in \Q^{n,1}$.
Then it follows that $L_a (v)=v$ for all $v \in V_a$, see Section \ref{lmaig}. In particular, $L_i (v)=v$ for all $v \in V_i$, that is, $L_i (v)=v$ for all $v \in V_\C.$
We know that projectivization of this semi-linear map $L$ is equal to projectivization of the linear map $L^\prime (v)=av.$ We have that $L^\prime \circ L^\prime  = - I$,
where $I$ is the identity map of $\Q^{n,1}$. This implies that projectivization of $L^\prime_a$ is an isometric involution of $\D$ whose fixed point set equals to $\D_a$.

\medskip

Since $ik=-ki$,  we have that $L^\prime_i \circ L^\prime_k = - L^\prime_k \circ L^\prime_i$. It follows that projectivizations of $L^\prime_i$ and $L^\prime_k$ are commuting involutions whose fixed points sets
are $\D_i =\D_\C$ and $\D_k$.

\medskip
Let again $M=\D_\C$ and $S= \D_\R$.  We define $N$ to be $\D_k$. Then $M \cap N = S.$  Also, let $I_M$ be projectivization of $L^\prime_i$, and let $I_N$ be projectivization of $L^\prime_k$.
Therefore, we get that $I_M$ and $I_N$ are geodesic reflections of ${\rm H^{n}_{\Q}}$  with respect to $M$ and $N$ respectively such that $I_M(N)=N$, $I_N(M)=M$,  $I_M$ and $I_N$ commute.
\end{proof}

\medskip

We remark that $I_M$ acts on $N$ as a geodesic reflection with respect to $S \subset N$, and  $I_N$ acts on $M$ as a geodesic reflection with respect to $S \subset M.$

\subsubsection{Fan decomposition of bisectors in quaternionic hyperbolic space}

In this section, we show that any bisector $B$ in quaternionic hyperbolic space ${\rm H^{n}_{\Q}}$ is the union of totally geodesic submanifolds in ${\rm H^{n}_{\Q}}$
isometric to ${\rm H^{n}_{\C}}$, all passing through a point $o$ in ${\rm H^{n}_{\Q}}$, that is,

$$B=\bigcup_{\alpha}{N_\alpha}(o),$$
where $N_\alpha (o)$ is a totally geodesic submanifolds in ${\rm H^{n}_{\Q}}$
isometric to ${\rm H^{n}_{\C}}$ passing through a point $o$ in ${\rm H^{n}_{\Q}}$.

\medskip

This decomposition in somewhat is similar to the Goldman meridional decomposition
of bisectors in complex hyperbolic space. We call  such decomposition a {\sl fan decomposition} of a bisector in  ${\rm H^{n}_{\Q}}$.

\medskip

We call the point $o  \in {\rm H^{n}_{\Q}}$ the {\sl center} of a fan decomposition of $B$ and the submanifolds $N_\alpha (o)$  the {\sl complex blades}  of a fan decomposition of $B$.

\medskip

A fan decomposition of a bisector $B$ is defined by its center $o$ and its complex blades.
\medskip

First, we define the center and the complex blades of a fan decomposition of a bisector.

\medskip

Let $B$ be a bisector in ${\rm H^{n}_{\Q}}$. Let $\Sigma$ and $\sigma$ be its quaternionic and real spine respectively.

\medskip

The center of a fan decomposition of $B$ is any point $o \in \sigma$.

\medskip

The complex blades of a fan decomposition of $B$  passing through a point $o$ are constructed as follows.

\medskip

Let $\gamma$ be the unique real geodesic in $\Sigma$ passing through the point $o$ orthogonal to $\sigma$. Let $p_1$, $p_2$  be distinct points in $\gamma$
symmetric with respect to $o$. Then it is clear that $B=B(p_1,p_2)$. Let $M$ be a totally geodesic submanifold of ${\rm H^{n}_{\Q}}$ isometric to ${\rm H^{n}_{\C}}$ containing $\gamma$.
Since $M$ is totally geodesic, we have that $B_M = M \cap B$ is a bisector in $M$.
We denote by $\Sigma_M$ the complex spine of $B_M$ and by $\sigma_M$ the real spine of $B_M$.
Then $\Sigma_M$ is a totally geodesic submanifold of $\Sigma$ isometric to
${\rm H^{1}_{\C}}$ containing $\gamma$, and $\sigma_M$ is a real geodesic in $\Sigma_M$ passing through the point $o$ orthogonal to $\gamma$. By applying
Goldman's meridianal decomposition, we have that there exists a totally geodesic submanifold $S \subset B_M $ isometric to ${\rm H^{n}_{\R}}$ containing $\sigma_M$,
$S$ is a meridian of $B_M $. Let us choose any such $S$ and  fix it. It follows from Theorem \ref{ortg1} that there exists a totally geodesic submanifold $N$
isometric to ${\rm H^{n}_{\C}}$ such that $ S=M \cap N$ and $N$ is orthogonal to $M$  along $S$. We define such a submanifold $N$ to be a {\sl complex blade} of a fan decomposition of $B$.

\medskip

It is seen that a complex blade of a fan decomposition of a bisector $B$ centered at $o$ is defined by $o$, by a totally geodesic submanifold $M$ of ${\rm H^{n}_{\Q}}$ isometric to ${\rm H^{n}_{\C}}$ containing $\gamma$,
and by a meridian of the bisector $B_M $.

\medskip

We will need the following fact proved in \cite{AG}.

\begin{prop} 
Let $p=(p_1, p_2, p_3)$ be an ordered triple of distinct points, $p_i \in \P\Q^{n}$. Then there exists a projective submanifold $W \subset \P\Q^{n}$ of complex type of complex dimension $2$ passing through the points $p_i$, that is, $p_i \in W$, $i=1,2,3.$ Moreover, this submanifold $W$ can be chosen, up to the action of ${\rm PU}(n,1; \Q)$, to be the canonical complex submanifold $ \P\C^{2} \subset \P\Q^{n}$.
\end{prop}

\begin{theorem}\label{fan1} Let $B$ and $o \in \sigma$ as above. Then $B$ is the union of its complex blades passing through the point $o$.
\end{theorem}

\begin{proof} Let $\gamma$ be the unique real geodesic in $\Sigma$ passing through the point $o$ orthogonal to $\sigma$. Let $p_1$, $p_2$  be distinct points in $\gamma$
symmetric with respect to $o$. Then it is clear that $B=B(p_1,p_2).$

\medskip
Let $M$ be a totally geodesic submanifold of ${\rm H^{n}_{\Q}}$ isometric to ${\rm H^{n}_{\C}}$ containing $\gamma$. Then $B_M = M \cap B$ is a bisector in $M$.
We denote by $\Sigma_M$ the complex spine of $B_M$ and by $\sigma_M$ the real spine of $B_M$. Then $\Sigma_M$ is a totally geodesic submanifold of $\Sigma$ isometric to
${\rm H^{1}_{\C}}$ containing $\gamma$, and $\sigma_M$ is a real geodesic in $\Sigma_M$ passing through the point $o$ orthogonal to $\gamma$. By applying
Goldman's meridianal decomposition, we have that there exists a totally geodesic submanifold $S \subset B_M $ isometric to ${\rm H^{n}_{\R}}$ containing $\sigma_M$,
$S$ is a meridian of $B_M $. Let us choose any such $S$ and  fix it. Let $N$ be a complex blade defined by $M$ and $S$, that is, $S=M \cap N$ and $N$ is orthogonal to $M$  along $S$.

\medskip
Let $I_N$ denote geodesic reflection of ${\rm H^{n}_{\Q}}$ with respect to $N$. Then it follows from Theorem \ref{ortg1} and Theorem 2.14 that $M$ is invariant with respect to $I_N$, and, moreover,
$I_N(p_1)=p_2$, $I_N(p_2)=p_1$.

\medskip
We first show that $N \subset B$. It suffices to show that for any $p \in N$ we have that $d(p,p_1)=d(p,p_2).$ But this is clear because

$$d(p,p_1)= d(I_N(p), I_N(p_1))= d(p,p_2).$$

\medskip

Conversely, we will show that each $p \in B$ lies in a complex blade containing the point $o$.

\medskip

Let $p \in B$. Then it follows from Proposition 2.3 that  there exists a totally geodesic submanifold $K$ of ${\rm H^{n}_{\Q}}$ isometric to ${\rm H^{n}_{\C}}$ passing through the points $p$, $p_1$, $p_2$.
Then $B_K = K \cap B$ is a bisector in $K$. Since $p \in B$, it follows that $p \in B_K$. By applying
Goldman's meridianal decomposition, we have that there exists a meridian $S_K$ of $B_K$ containing $p$.  By applying Theorem \ref{ortg1}, we get there exists a totally geodesic submanifold  $L$ of ${\rm H^{n}_{\Q}}$ isometric to ${\rm H^{n}_{\C}}$
orthogonal to $K$ along $S_K$, that is, $L$ is a complex blade containing $p$. It follows from the above that $L \subset B$. Since, by construction, $p \in L$, the result follows.

\medskip

It is seen that two distinct complex blades intersect at the point $o$.
\end{proof}

\medskip

As a corollary of this theorem we have the following:

\begin{theorem} 
Bisectors in quaternionic  hyperbolic space  ${\rm H^{n}_{\Q}}$ are star-like with respect to any point in its real spine.
\end{theorem}

\begin{proof} 
Let $B$ be a bisector in ${\rm H^{n}_{\Q}}$ and $o$ be a  point in $\sigma$. Let $p$ be an arbitrary point in $B$. By applying Theorem \ref{fan1}, we have that
there exists a totally geodesic submanifold $M \subset B$ isometric to  ${\rm H^{n}_{\C}}$ containing $o$ and $p$. This implies that the geodesic passing through $o$ and $p$
lies in $B$.
\end{proof}

%---------------------------------------------------------------------

\section{Orthogonal projection onto totally geodesic submanifolds of complex type of maximal dimension and complex hyperbolic packs}
 
In this section, we obtain a formula for the orthogonal projection onto totally geodesic submanifolds of ${\rm H^{n}_{\Q}}$ isometric to ${\rm H^{n}_{\C}}$. Using this formula we will show that the inverse image of a bisector with respect to this projection is not a bisector in ${\rm H^{n}_{\Q}}.$ 

\subsection{Orthogonal projection onto totally geodesic submanifolds of complex type of maximal dimension}

\begin{lemma}
Let $N$ be a totally geodesic reflective submanifold of a hyperbolic space $M$. Let $\phi: M\to M$ be the geodesic reflection of $M$ with respect to $N$, and \ $\Pi:M\to N$ be the orthogonal projection of $M$ onto $N$. Then $\Pi(p)$ is the midpoint of the geodesic segment defined by $p$ and $\phi(p)$ for all $p\in M\setminus N.$ 
\end{lemma}
\begin{proof}
Given $p\in M \setminus N$, let $\xi\in T_{\Pi(p)}M$ be a unitary vector normal to $N$ at $\Pi(p)$, such that $\gamma(s)=\exp_{\Pi(p)}(s\xi)$, $s\in\R,$ is the unique geodesic in $M$ passing through $p$ and $\Pi(p)$. Then
$$\gamma(-s)=\exp_{\Pi(p)}(-s\xi)=\exp_{\phi(\Pi(p))}(d\phi_{\Pi(p)}(s\xi))=\phi(\exp_{\Pi(p)}(s\xi))=\phi(\gamma(s)),$$
for all $s\in\R$.
\end{proof}

\medskip

\begin{rmk} It is easy to see that $\Pi(p)=m$ is the unique fixed point of the composition
$\phi$ and $\psi$, where $\psi$ is the geodesic reflection of $M$ with respect to the reflective submanifold  $N^\bot_m$ orthogonal to $N$ at $m$.

\end{rmk}

In what follows, we use the ball model and the notations in Section 2.2.4, see also Theorem 2.14. That is,

\medskip

$$\D=\{x= (x_1,x_2, \ldots, x_n) \in \Q^n :  |x_1|^2 + |x_2|^2 + \ldots + |x_n|^2 < 1   \},$$

\medskip

$$\D_\C=\{z= (z_1,z_2, \ldots, z_n) \in \C^n :  |z_1|^2 + |z_2|^2 + \ldots + |z_n|^2 < 1   \}.$$

\medskip

It is clear that for each $x\in\Q^n$ there exist unique vectors $u$, $v\in\C^n$ such that
$$x=u + v j.$$

\medskip

Using this, we identify $\D_\C$ with $\D_u$:

$$\D_u= \{u=(u_1,u_2, \ldots, u_n) \in \C^n :  |u_1|^2 + |u_2|^2 + \ldots + |u_n|^2 < 1   \}.$$

\medskip

We define 
$$\hat{x}:=u - v j.$$

\medskip

\begin{rmk}
It follows from Section 1.2.5 that the map $x\mapsto\hat{x}$, $x\in\D,$ coincides with the geodesic reflection $L_i$ with respect to $\D_{\C} = \D_u.$
\end{rmk}

\begin{lemma}
Let $\R_{-}$ be the set of negative real numbers, and let $p\in\Q \setminus\R_{-}.$ Then 
$\lambda^2=p$ for $\lambda\in\Q$ if and only if
$$\lambda=\pm\left(\sqrt{\frac{|p|+\re p}{2}~}+\sqrt{\frac{|p|-\re p}{2}~}\,\omega\right),$$
where $\omega=0$ if $p\in\R \setminus\R_{-}$ or $\omega=\frac{\im p}{|\im p|}$ otherwise. 
\end{lemma}
\begin{proof}
The proof follows along the lines in \cite{Alf}, Section 1.2. On the other hand, the reader can verify the result by calculating $\lambda^2$ directly for $\lambda$ given in the lemma.
It is understood that all square roots of positive numbers are 
taken with the positive sign. 
\end{proof}

\medskip

Let $M$ be a non-isotropic vector in $\Q^{n,1}$, that is, $\herm{M,M} \neq 0.$ Since $\herm{M,M}$ is real, the following map

$$I=I(M): Z \rightarrow -Z + 2 M \frac{\herm{M,Z}}{\herm{M,M}},$$
where $Z \in \Q^{n,1}$, is well-defined.

\medskip

It is easy to see  that $I: \Q^{n,1}  \rightarrow   \Q^{n,1}$ is a right linear map, that is,

$$I(Z_1 \lambda_1 + Z_2 \lambda_2)= I(Z_1) \lambda_1 + I(Z_2)\lambda_2$$
for all $Z_1, Z_2 \in \Q^{n,1}$ and for all $\lambda_1, \lambda_2 \in \Q.$ 

\medskip

A direct computation shows  that for all $X,Y \in  \Q^{n,1}$ we have that $\herm{I(X),I(Y)} =\herm{X,Y}$. Indeed,

\begin{align*}
\herm{I(X),I(Y)}&=\left\langle-X+2M\dfrac{\herm{M,X}}{\herm{M,M}},-Y+2M\dfrac{\herm{M,Y}}{\herm{M,M}}\right\rangle\\
                &=\herm{X,Y}-2\herm{X,M}\dfrac{\herm{M,Y}}{\herm{M,M}}-2\overline{\dfrac{\herm{M,X}}{\herm{M,M}}}\herm{M,Y}
               +4\overline{\dfrac{\herm{M,X}}{\herm{M,M}}}\herm{M,M}\dfrac{\herm{M,Y}}{\herm{M,M}}\\                
              &=\herm{X,Y}-2\herm{X,M}\dfrac{\herm{M,Y}}{\herm{M,M}}-2\dfrac{\herm{X,M}}{\herm{M,M}}\herm{M,Y}  +4\dfrac{\herm{X,M}\herm{M,Y}}{\herm{M,M}} 
                &=\herm{X,Y}.
\end{align*}

\medskip

Also, it is easy to show that the map $I(M)$ is an involution.

\medskip

\begin{lemma} Let $I(X,Y)=I(M)=I$, where $M=X+Y$.  Then $I(X)=Y$ if and only if 

\begin{itemize}
\item $\herm{X,X}=\herm{Y,Y},$
\item $\herm{X,Y}\in\R$.
\end{itemize}
\end{lemma}
\begin{proof} 
Since $I^2=\mathrm{Id}$, we have that $I(X)=Y$ if and only if $I(Y)=X$. Then
\begin{align*}
I(X)=Y &\iff Y=-X+2M\dfrac{\herm{M,X}}{\herm{M,M}}\\
       &\iff M=2M\dfrac{\herm{M,X}}{\herm{M,M}}\\
       &\iff \herm{M,M}=2\herm{M,X}\\
       &\iff \herm{X,X}+\herm{X,Y}+\herm{Y,X}+\herm{Y,Y}=2\herm{X,X}+2\herm{Y,X}\\
       &\iff \herm{X,Y}-\herm{Y,X}=\herm{X,X}-\herm{Y,Y}\\
       &\iff \herm{X,Y}-\overline{\herm{X,Y}}=\herm{X,X}-\herm{Y,Y}\\
       &\iff 2\,\im\herm{X,Y}=\herm{X,X}-\herm{Y,Y}\\
       &\iff \herm{X,Y}\in\R \text{ and } \herm{X,X}=\herm{Y,Y}\\
\end{align*}

\end{proof}
\medskip

Now let $M=X+Y$, where $X$ and $Y$ are negative vectors in $\Q^{n,1}$, such that $\herm{X,Y}<0.$ Then $M$ is negative. Indeed, $M$ is negative if and only if $ \herm{M,M}<0$. We have that 

$$\herm{M,M}=\herm{X,X}+2\re\herm{X,Y}+\herm{Y,Y}.$$

\medskip

\begin{co} 
Let $X$, $Y$, $M$ be vectors in $\Q^{n,1}$ as in  Lemma 3.3. Let $x=\pi(X)$, $y=\pi(Y)$, $m=\pi(M)$ be their projectivizations in ${\rm H^{n}_{\Q}}$. Let $i=i(X,Y)$ denote the projectivization of 
$I(X,Y)$. Then $i=i(X,Y)$ is an isometry of ${\rm H^{n}_{\Q}}$ such that $i(x)=y$, $i(y)=x$ and $i(m)=m$.
\end{co}

\medskip

We call this isometry $i=i(X,Y)$ a {\sl geodesic inversion (reflection)} with respect to the point $m$.

\medskip

\begin{co} Let $x$ and $y$ be two points in  ${\rm H^{n}_{\Q}}$. Then there exists an isometric involution $i=i(x,y)$ of ${\rm H^{n}_{\Q}}$ such that $i(x)=y$ and $i(y)=x$.
\end{co}

\begin{proof} Let $x$ and $y$ be two points in  ${\rm H^{n}_{\Q}}$. Let $X^\prime$ and $Y^\prime$ be negative vectors in 
$\Q^{n,1}$ representing  $x$ and $y$. By multiplying $X^\prime$ and $Y^\prime$ from the right  by suitable elements of $\Q$, we obtain vectors $X,Y \in \Q^{n,1}$ satisfying:

$$\herm{X,X}=\herm{Y,Y},  \ \herm{X,Y}<0. $$

Then the proof follows from Corollary 3.1
\end{proof}

\medskip
It is easy to verify that  the midpoint $m$ of $x$ and $y$ with respect to the Bergman metric is represented by the vector $M$ as in Corollary 3.1, moreover, $m$ is a unique fixed point of the involution $i(x,y)$.

\medskip

Let $z,w$ be two vectors in $\Q^n$. We recall, see Section 1.2.2, that   $\hherm{z,w}$ denotes the standard positive definite Hermitian product:

$$\hherm{z,w}= \overline{z}_1 w_1 + \ldots, + \overline{z}_n w_n.$$

\medskip

\begin{theorem} The orthogonal projection $\Pi_u:\D\to\D_u$ is given by
\begin{equation}
\Pi_u(x)=
\frac{\hat{x}+x}{2}-\frac{\hat{x}-x}{2}~\frac{\im\hherm{\hat{x},x}}{|1-\hherm{\hat{x},x}|+1-\re\hherm{\hat{x},x}},
\end{equation}
for all $x\in\D$.
\end{theorem}
\begin{proof}
Take $x\in\D-\D_u$. If $X=(x,1)^t$, let $\hat{X}=(\hat{x},1)^t$. It is easy to verify that the vectors $X\hat{\lambda}$ and $\hat{X}\lambda$, where 
$$\lambda^2=-\herm{\hat{X},X}=1-\hherm{\hat{x},x}$$ 
satisfy all the conditions of Lemma 3.3, that is, 

\begin{itemize}
\item $X\hat{\lambda}$ and $\hat{X}\lambda$ are negative, 
\item $\herm{X\hat{\lambda},\ \hat{X}\lambda}<0,$
\item $\herm{X\hat{\lambda}, \ X\hat{\lambda}} =\herm{\hat{X}\lambda, \ \hat{X}\lambda}.$
\end{itemize}

\medskip 

Let $Y=X\hat{\lambda}$ +$\hat{X}\lambda.$ Then it follows from Lemma 3.1, Remark 3.1, and Corollary 3.1 that
$$\Pi_u(x)=\mathrm{mid}\,(x,\hat{x})=y,$$
where $y$ is the projectivization of $Y$.

\medskip

Now, set $x=u+v j$ and $\lambda=\alpha+\beta j$, where $u$, $v\in\C^n$ and $\alpha$, $\beta\in\C$. We have that the last coordinate of $Y$ is equal to $\lambda+\hat{\lambda}=2\alpha$. Below we will show that $\alpha\neq 0$. Therefore,

$$y=u+\dfrac{\overline{\beta}}{\alpha}\,v.$$

\medskip

To find $\lambda$, $\alpha$ and $\beta$, note that $\hherm{\hat{x},x}=\hherm{u,u}-\hherm{v,v}+2\hherm{u,v}\,j$. So,
\begin{equation}
\label{lambda}
\re(1-\hherm{\hat{x},x})=1-\hherm{u,u}+\hherm{v,v}>2\hherm{v,v}\geq0,
\end{equation}
because $1>\hherm{x,x}=\hherm{u,u}+\hherm{v,v}$. Then it follows from  Lemma 3.2 that

$$\lambda=\pm\left(\sqrt{\frac{|1-\hherm{\hat{x},x}|+\re (1-\hherm{\hat{x},x})}{2}~}+\sqrt{\frac{|1-\hherm{\hat{x},x}|-\re(1-\hherm{\hat{x},x})}{2}~}\,\omega\right),$$
where $\omega=0$ if $\hherm{u,v}=0$ or $\omega=-\frac{\hherm{u,v}}{|\hherm{u,v}|}\,j$ otherwise. Therefore,

$$\alpha=\pm\sqrt{\frac{|1-\hherm{\hat{x},x}|+\re(1-\hherm{\hat{x},x})}{2}}$$
and

$$\beta=\pm
\begin{cases}
0&,\text{ if } \hherm{u,v}=0\\
-\frac{\hherm{u,v}}{|\hherm{u,v}|}\sqrt{\frac{|1-\hherm{\hat{x},x}|-\re(1-\hherm{\hat{x},x})}{2}}&,\text{ if } \hherm{u,v}\neq 0
\end{cases}.$$

\medskip

We see from (\ref{lambda}) that $\alpha\neq 0.$

\medskip

Then we have

\begin{align*}
\frac{\bar{\beta}}{\alpha}&=
\begin{cases}
0&,\text{ if } \hherm{u,v}=0\\
-\frac{\hherm{v,u}}{|\hherm{u,v}|}\sqrt{\frac{|1-\hherm{\hat{x},x}|-\re(1-\hherm{\hat{x},x})}{|1-\hherm{\hat{x},x}|+\re(1-\hherm{\hat{x},x})}}&,\text{ if } \hherm{u,v}\neq 0
\end{cases}\\
&=
\begin{cases}
0&,\text{ if } \hherm{u,v}=0\\
-\frac{\hherm{v,u}}{|\hherm{u,v}|}\sqrt{\frac{|1-\hherm{\hat{x},x}|-\re(1-\hherm{\hat{x},x})}{|1-\hherm{\hat{x},x}|+\re(1-\hherm{\hat{x},x})}\cdot\frac{|1-\hherm{\hat{x},x}|+\re(1-\hherm{\hat{x},x})}{|1-\hherm{\hat{x},x}|+\re(1-\hherm{\hat{x},x})}}&,\text{ if } \hherm{u,v}\neq 0
\end{cases}\\
&=
\begin{cases}
0&,\text{ if } \hherm{u,v}=0\\
-\frac{\hherm{v,u}}{|\hherm{u,v}|}\frac{|\im(1-\hherm{\hat{x},x}|}{|1-\hherm{\hat{x},x}|+\re(1-\hherm{\hat{x},x})}&,\text{ if } \hherm{u,v}\neq 0
\end{cases}\\
&=
\begin{cases}
0&,\text{ if } \hherm{u,v}=0\\
-\frac{\hherm{v,u}}{|\hherm{u,v}|}\frac{2|\hherm{u,v}|}{|1-\hherm{\hat{x},x}|+\re(1-\hherm{\hat{x},x})}&,\text{ if } \hherm{u,v}\neq 0
\end{cases}\\
&=
\begin{cases}
0&,\text{ if } \hherm{u,v}=0\\
\frac{-2\hherm{v,u}}{|1-\hherm{\hat{x},x}|+\re(1-\hherm{\hat{x},x})}&,\text{ if } \hherm{u,v}\neq 0
\end{cases}\\
&=\frac{-2\hherm{v,u}}{|1-\hherm{\hat{x},x}|+ {\rm Re}(1-\hherm{\hat{x},x})}
\end{align*}
Finally, since $u=\dfrac{\hat{x}+x}{2}$ e $v=\dfrac{\hat{x}-x}{2}\,j$, we get
$$\hherm{u,v}=\frac{1}{4}(-\hherm{\hat{x},x}+\hherm{x,\hat{x}})\,j=-\frac{1}{2}{\rm  Im}\hherm{\hat{x},x}\,j,$$
because $\hherm{\hat{x},\hat{x}}=\hherm{x,x}$. Then
$$-2\hherm{v,u}=\overline{{\rm Im}\hherm{\hat{x},x}\,j}=(-j)(-{\rm Im} \hherm{\hat{x},x})=j\,{\rm Im}\hherm{\hat{x},x}.$$
Therefore,
$$\dfrac{\overline{\beta}}{\alpha}\,v=v\,\dfrac{\overline{\beta}}{\alpha}=\dfrac{\hat{x}-x}{2}\,j\dfrac{j\,{\rm Im}\hherm{\hat{x},x}}{|1-\hherm{\hat{x},x}|+\re(1-\hherm{\hat{x},x})}=-\,\dfrac{\hat{x}-x}{2}\,\dfrac{{\rm Im}\hherm{\hat{x},x}}{|1-\hherm{\hat{x},x}|+1-{\rm Re}\hherm{\hat{x},x}}$$
\end{proof}

%%%%%%%%%
\subsection{Complex hyperbolic packs in quaternionic hyperbolic space}

Let $S$ be a totally geodesic submanifold of $\H^n_{\Q}$. Let $\Pi= \Pi_S: \H^n_{\Q}  \to S$
be the orthogonal projection onto  $S$. If $B$ is a bisector in $\H^n_{\Q}$, then examples show that the intersection $B'=S\cap B$ tends to be complicate. There are examples when $B'$ is a bisector in $S$, and there are examples when $B'$ is not even a part of a bisector in $S$. Also, the relation between bisectors in $\H^n_{\Q}$ and their orthogonal projections onto $S$ is not clear. A good situation happens in the Mostow decomposition and its generalization given by Theorem 2.13 where $S$ is isometric to  $\H^k_{\Q}$, $k=1, \ldots ,n-1$.  Next we show that when $S$ is isometric to $\H^n_{\C}$
the situation is bad: there exists a bisector $B \subset\H^n_{\Q}$ such that $B'$ is a bisector in $S$ but $B$ is not equal to $\Pi^{-1}(B')$. 

\medskip

\noindent \textbf{Example}. In what follows, we use the ball model for $\H^n_{\Q}$ and the notations in Section 3.1. For simplicity, we consider the case $n=2$.

\medskip

Let $p_1=(0,-\delta)$ and $p_2=(0,\delta)$, for some $0<\delta<1$. It is easy to see that the bisector $B=B(p_1,p_2)$ in $\D^2$ is equal to $\{x\in\D^2:\re(x_2)=0\}$. Take $S =\D^2_{\C}=\{z\in\D^2:\, z\in\C^2\}$. Also we have that the bisector $B'=B'(p_1,p_2)$ in $S$ is equal to $\{z\in\D^2_{\C}:\re(z_2)=0\}$. Then it follows that $B'=B\cap S$.

\medskip

Let $x=(r+rj,ri+rj)$, for some $0<r<1/2$. An easy computation shows that $x\in B$. If we write $x=u+vj$, then

$$u=(r,ri)\text{  and }v=(r,r).$$ 

Therefore, $\hherm{u,u}=\hherm{v,v}=2r^2$ and $\hherm{u,v}=r^2(1-i)$. Then we get that

$$\hherm{\hat{x},x}=\hherm{u,u}-\hherm{v,v}+2\hherm{u,v}\,j=2r^2(1-i)\,j$$

and

$$\re(1-\hherm{\hat{x},x})=1\text{ and }|1-\hherm{\hat{x},x}|=\sqrt{1+8r^4}.$$

This implies that

$$\frac{\bar{\beta}}{\alpha}=\frac{-2\hherm{v,u}}{|1-\hherm{\hat{x},x}|+\re(1-\hherm{\hat{x},x})}=\frac{-2r^2(1+i)}{\sqrt{1+8r^4}+1}.$$

By applying Theorem 3.1, we get that $\Pi(x)=y =u+\bar{\beta}/\alpha\,v$. Then it follows that $y\not\in B'$ because

$$\re(y_2)=\re(u_2)+\re(\bar{\beta}/\alpha\,v_2)=\re(ir)+\re(\bar{\beta}/\alpha\,r)=$$

$$r\re(\bar{\beta}/\alpha)=\frac{-2r^3}{\sqrt{1+8r^4}+1}\neq0.$$

\medskip

It follows from this example that the inverse image of $B'=B'(p_1,p_2)$ under the orthogonal projection  $\Pi$ onto $S$ is not a bisector in $\D^2$. Since all bisectors in $\H^2_{\C}$ 
are equivalent with respect to the action of the isometry group of $\H^2_{\C}$, we get using simple geometric arguments that this fact is true for any bisector in  $S$. 

\medskip

We remark that a similar situation happens in complex hyperbolic space $\H^2_{\C}$ when $S$ is a totally geodesic submanifold in $\H^2_{\C}$ isometric to $\H^2_{\R}$,  a totally real geodesic submanifold of $\H^2_{\C}$, see Parker-Platis \cite{PP}.

\medskip

Let $S$ be a totally geodesic submanifold of $\H^n_{\Q}$ isometric to $\H^n_{\C}$ and $\Pi$ be the orthogonal projection onto $S$. Let $B$ be a bisector in $S$. Following Parker-Platis \cite{PP}, we call the inverse image $P=\Pi^{-1}(B)$ a {\sl complex hyperbolic pack} in $\H^n_{\Q}$ defined by $B$.

\medskip

Since orthogonal projection $\Pi  :  {\rm H^{n}_{\Q}} \rightarrow S$ is a real analytic fibration, we have:

\begin{co} 
A complex hyperbolic pack $P$  in ${\rm H^{n}_{\Q}}$ is a real analytic hypersurface in ${\rm H^{n}_{\Q}}$ diffeomorphic to $\R^{4n -1}$.
\end{co}

\medskip

We call $S$ a {\sl complex spine} of $P$ and a bisector $B \subset S$ a {\sl spine} of $P$.

\medskip

Let $\Pi :  {\rm H^{n}_{\Q}} \rightarrow S$ be orthogonal projection onto $S$. It follows from the construction of $P$ that the fibers of $\Pi$ are totally
geodesic submanifolds of ${\rm H^{n}_{\Q}}$ isometric to  ${\rm H^{n}_{\C}}$, see Section 2.2.1. We call these fibers {\sl slices} of the pack $P$. It is clear that any two slices of  $P$ are ultraparallel.

\medskip

In \cite{PP}, the authors used the packs in complex hyperbolic geometry to construct fundamental domains (polyhedra) for complex hyperbolic quasi-Fuchsian groups. We hope that complex hyperbolic packs in ${\rm H^{n}_{\Q}}$ introduced here will be useful for constructing fundamental polyhedra for discrete subgroups of isometries of ${\rm H^{n}_{\Q}}.$  As an example, we consider the following construction:

\medskip

\noindent \textbf{Example}. Let $G$ be a discrete group in  ${\rm PU}(n,1; \Q)$ leaving invariant a totally geodesic submanifold $S$ of ${\rm H^{n}_{\Q}}$ isometric to  ${\rm H^{n}_{\C}}$. Let $F$ be a fundamental polyhedron for the action of $G$ in $S$ bounded by bisectors in $S$ (for instance, $F$ is the Dirichlet polyhedron in $S$ with the center at a point $o$ in $S$).
Then, the polyhedron $\tilde{F}$ in ${\rm H^{n}_{\Q}}$ bounded by the complex hyperbolic packs defined by all the bisectors  which form the boundary of $F$ is a fundamental polyhedron for $G$ in ${\rm H^{n}_{\Q}}$.
The proof of this follows from simple geometric arguments using that the orthogonal projection onto $S$ is equivariant with respect to the action of the stabilizer of $S$ in  
${\rm PU}(n,1; \Q)$. It is easy to see that the combinatorial structure of  $\tilde{F}$ is more simple than the combinatorial structure of  the Dirichlet polyhedron with the center at a point $o$ in $S$ for the action of $G$ in  ${\rm H^{n}_{\Q}}$. 

%%%%%%%%%%%
%\section*{Declarations}
%\textbf{Data availability} Not applicable.\\
%\\
\textbf{Conflict of interest} The authors have no conflict of interest to declare that are relevant to this article.

%--------------------------------------------------------------------

\end{document}